\documentclass[12pt,reqno]{article}

\usepackage[usenames]{color}
\usepackage{amssymb}
\usepackage{amsmath}
\usepackage{amsthm}
\usepackage{amsfonts}
\usepackage{amscd}
\usepackage{graphicx}

\usepackage[colorlinks=true,
linkcolor=webgreen,
filecolor=webbrown,
citecolor=webgreen]{hyperref}

\definecolor{webgreen}{rgb}{0,.5,0}
\definecolor{webbrown}{rgb}{.6,0,0}

\usepackage{color}
\usepackage{fullpage}
\usepackage{float}

\usepackage{graphics}
\usepackage{latexsym}

\begin{document}

\theoremstyle{plain}
\newtheorem{theorem}{Theorem}
\newtheorem{corollary}[theorem]{Corollary}
\newtheorem{lemma}{Lemma}
\newtheorem{example}{Example}
\newtheorem*{remark}{Remark}

\begin{center}
\vskip 1cm
{\LARGE\bf Some notes on an identity of Frisch \\ }

\vskip 1cm

{\large
Kunle Adegoke \\
Department of Physics and Engineering Physics \\ Obafemi Awolowo University, 220005 Ile-Ife, Nigeria \\
\href{mailto:adegoke00@gmail.com}{\tt adegoke00@gmail.com}

\vskip 0.2 in

Robert Frontczak \\
Independent Researcher \\ Reutlingen,  Germany \\
\href{mailto:robert.frontczak@web.de}{\tt robert.frontczak@web.de}
}
\end{center}

\vskip .2 in

\begin{abstract}
In this note, we show how a combinatorial identity of Frisch can be applied to prove and generalize some
well-known identities involving harmonic numbers. We also present some combinatorial identities involving
odd harmonic numbers which can be inferred straightforwardly from our results.
\end{abstract}



\bigskip

\section{Preliminaries}

The binomial coefficients are defined, for non-negative integers $n$ and $m$, by
\begin{equation*}
\binom {n}{m} =
\begin{cases}
\dfrac{n!}{m!(n - m)!}, & \text{$n\geq m$};\\
0, & \text{$n<m$}.
\end{cases}
\end{equation*}
More generally, for complex numbers $r$ and $s$, they are defined by
\begin{equation*}
\binom {r}{s} = \frac{\Gamma (r+1)}{\Gamma (s+1) \Gamma (r-s+1)},
\end{equation*}
where the Gamma function, $\Gamma(z)$, is defined for $\Re(z)>0$ by the integral \cite{Srivastava}
\begin{equation*}
\Gamma (z) = \int_0^\infty e^{- t} t^{z - 1}\,dt.
\end{equation*}
The function $\Gamma(z)$ can be extended to the whole complex plane by analytic continuation. 
It has a simple pole at each of the points $z=\cdots,-3,-2,-1,0$. 
The Gamma function extends the classical factorial function to the complex plane: $(z-1)!=\Gamma(z)$; 
thereby facilitating the computation of the binomial coefficients at non-integer and non-real values.
Closely related to the Gamma function is the psi or digamma function defined by $\psi(z)=\Gamma'(z)/\Gamma(z)$.
It possesses the infinite series expression \cite[p. 14]{Srivastava}
\begin{equation}\label{psi_expression}
\psi (z) = - \gamma + \sum_{k=0}^\infty \left ( \frac{1}{k+1}-\frac{1}{k+z} \right ),
\end{equation}
with $\gamma$ being the Euler-Mascheroni constant. 

Harmonic numbers, $H_x$, $x\in\mathbb C\setminus\mathbb Z^{-}$, defined by the relation
\begin{equation}\label{recurrence}
H_x  = H_{x - 1}  + \frac{1}{x},\quad H_0  = 0,
\end{equation}
are connected to the digamma function through the fundamental relation
\begin{equation}\label{eq.cx26aju}
H_x=\psi(x + 1) + \gamma.
\end{equation}
When each $x$ is a positive integer, say $n$, we have the sequence of harmonic numbers, $(H_n)_{n\in\mathbb Z^+}$ and the recurrence relation~\eqref{recurrence} then gives
\begin{equation}
H_n =\sum_{k=1}^n \frac{1}{k},\quad H_0=0.
\end{equation}

The following combinatorial identity is attributed to Frisch \cite{Frisch}:
\begin{equation}\label{Id_Frisch}
\sum_{k=0}^n (-1)^k \frac{\binom{n}{k}}{\binom{b+k}{c}} = \frac{c}{n+c}\frac{1}{\binom{n+b}{b-c}},\quad \text{$n\in\mathbb Z^+$ and $b,\, c,\, b - c\in\mathbb C\setminus\mathbb Z^{-}$}.
\end{equation}
It is listed in Gould's compendium \cite{Gould} as Identity 4.2 and was used recently by Gould and Quaintance \cite{Gould2}
to prove a new binomial transform identity. Abel \cite{Abel} gave a short proof for Frisch's identity and studied its infinite variant. 
In this note, we report how Frisch's identity~\eqref{Id_Frisch} can be applied to prove and generalize some well-known identities involving harmonic numbers $H_n$. We also present some combinatorial identities involving odd harmonic numbers $O_n$, which can be inferred straightforwardly from our results. Our line of approach consists essentially of utilizing the fact that derivatives of the generalized binomial coefficients yield harmonic numbers. This method is familiar and has appeared in some important earlier research of other mathematicians \cite{Chu1,Chu2,Wang,Wei}.

\section{Results}

% We start by pointing out that the restrictions placed on $b$ and $c$ in Frisch's identity can be weakened. 
% That is, the parameters $b$ and $c$ in his identity and hence in identities derived from it can, in general, 
% take on complex values, not just integers. 

\begin{theorem}\label{thm1}
For $n\in\mathbb Z^+$ and $b,\, c,\, b - c\in\mathbb C\setminus\mathbb Z^{-}$, we have
\begin{equation}\label{main_id1}
\sum_{k=0}^n (-1)^k \frac{\binom{n}{k}}{\binom{b+k}{c}}\left ( H_{k+b} - H_{k+b-c}\right ) 
= \frac{c}{n+c} \frac{H_{n+b}-H_{b-c}}{\binom{n+b}{n+c}}
\end{equation}
and
\begin{equation}\label{main_id1b}
\sum_{k = 0}^n (- 1)^k \binom{{n}}{k}\frac{c}{{k + c}}\frac{{H_{k + b} - H_{b - c}}}{{\binom{{k + b}}{{k + c}}}}  
= \frac{{H_{n + b} - H_{n + b - c} }}{{\binom{{b + n}}{c}}}.
\end{equation}
In particular, for $n\in\mathbb Z^+$ and $b\in\mathbb C\setminus\mathbb Z^{-}$ we have
\begin{equation}\label{main_id11}
\sum_{k=0}^n (-1)^k \frac{\binom {n}{k}}{\binom{b+k}{k}}\left ( H_{k+b} - H_{k}\right ) 
= \frac{b}{n+b} H_{n+b}
\end{equation}
and
\begin{equation}\label{main_id11b}
\sum_{k = 0}^n (- 1)^k \binom{{n}}{k}\frac{{H_{k + b} }}{{k + b}} = \frac{{H_{n + b} - H_n }}{{b\binom{{b + n}}{b}}}.
\end{equation}
\end{theorem}
\begin{proof}
In Frisch's identity \eqref{Id_Frisch} treat $b$ and $c$ as complex numbers and differentiate w.r.t. $b$ using
\begin{equation}\label{eq.u72pdmn}
\frac{d}{db} \binom{b+k}{c}^{-1} = \binom{b+k}{c}^{-1}\left (\psi(b+k+1-c)-\psi(b+k+1)\right )
\end{equation}
and
\begin{equation}\label{eq.lilpf1k}
\frac{d}{db} \binom{n+b}{n+c}^{-1} = \binom{n+b}{n+c}^{-1}\left (\psi(b+1-c)-\psi(b+n+1)\right ),
\end{equation}
where $\psi(z)$ is the psi function. Simplify, making use of the fundamental relation \eqref{eq.cx26aju}.
Identity~\eqref{main_id1b} is the binomial transform of~\eqref{main_id1}.
\end{proof}

\begin{corollary}
For $n\in\mathbb Z^+$ and $b\in\mathbb C\setminus\mathbb Z^{-}$, $b\ne 0$ we have
\begin{equation}\label{cor_id1}
\sum_{k=0}^n (-1)^k \frac{\binom{n}{k}}{(b+k)^2} = \frac{1}{n+1} \frac{H_{n+b}-H_{b-1}}{\binom{n+b}{n+1}}.
\end{equation}
\end{corollary}

Identity \eqref{cor_id1} generalizes the well-known identity
\begin{equation}
\sum_{k=0}^n (-1)^k \frac{\binom{n}{k}}{(k+1)^2} = \frac{H_{n+1}}{n+1},
\end{equation}
which can be proved directly via
\begin{align*}
(n+1) \sum_{k=0}^n \binom {n}{k} \frac{(-1)^k}{(k+1)^2} &= \sum_{k=0}^n \frac{n+1}{k+1} \binom {n}{k} \frac{(-1)^k}{k+1} \\
&= \sum_{k=0}^n \binom {n+1}{k+1} \frac{(-1)^k}{k+1} \\
&= \sum_{k=1}^{n+1} \binom {n+1}{k} \frac{(-1)^{k-1}}{k} \\
&= H_{n+1}.
\end{align*}

\begin{theorem}\label{thm2}
For $n\in\mathbb Z^+$ and $b,\, c,\, b - c\in\mathbb C\setminus\mathbb Z^{-}$ we have
\begin{equation}\label{main_id2}
\sum_{k=0}^n (-1)^{k+1} \frac{\binom{n}{k}}{\binom{b+k}{c}}\left ( H_{k+b-c} - H_{c}\right ) 
= \frac{1}{(n+c)\binom{n+b}{n+c}} \left (\frac{n}{n+c} + c (H_{n+c}-H_{b-c})\right ).
\end{equation}
In particular, for $n\in\mathbb Z^+$ and $b\in\mathbb C\setminus\mathbb Z^{-}$ we have
\begin{equation}\label{main_id22}
\sum_{k=0}^n (-1)^{k+1} \binom{n}{k} H_{k+b} = \frac{1}{n\binom{n+b}{n}}.
\end{equation}
\end{theorem}
\begin{proof}
Differentiate Frisch's identity \eqref{Id_Frisch} w.r.t. $c$ using
\begin{equation*}
\frac{d}{dc} \binom{b+k}{c}^{-1} = \binom{b+k}{c}^{-1}\left (\psi(c+1)-\psi(b+k+1-c)\right ).
\end{equation*}
\end{proof}

Identity \eqref{main_id22} generalizes the well-known identity
\begin{equation}
\sum_{k=0}^n (-1)^{k} \binom{n}{k} H_{k} = - \frac{1}{n},
\end{equation}
which is the binomial transform of the sequence $H_n$ (see \cite[p. 34]{Boyadzhiev}).

\begin{corollary}
For $n\in\mathbb N_0$ and $b\in\mathbb C\setminus\mathbb Z^{-}$ we have
\begin{equation}\label{eq.auhzvb5}
\sum_{k=1}^n (-1)^{k+1} \frac{\binom{n}{k}}{k\binom{k+b}{k}} = H_{n+b} - H_{b}.
\end{equation}
In particular, 
\begin{equation}
\sum_{k=1}^n (-1)^{k+1} \frac{\binom{n}{k}}{k\binom{k+n}{k}} = H_{2n} - H_{n}.
\end{equation}
\end{corollary}

\begin{corollary}
For $n\in\mathbb Z^+$ and $b\in\mathbb C\setminus\mathbb Z^{-}$ we have
\begin{equation}\label{cor_id2}
\sum_{k=0}^n (-1)^k \frac{\binom{n}{k}}{\binom{k+b}{k}} H_{k+b} = \frac{b}{n+b} H_{b} - \frac{n}{(n+b)^2}.
\end{equation}
\end{corollary}
\begin{proof}
Set $c=b$ in \eqref{main_id2} to get
\begin{equation*}
\sum_{k=0}^n (-1)^{k+1} \frac{\binom{n}{k}}{\binom{b+k}{b}}\left ( H_{k} - H_{b}\right ) 
= \frac{1}{n+b} \left (\frac{n}{n+b} + b H_{n+b}\right ).
\end{equation*}
Use the fact that
\begin{equation}\label{eq.twi4acl}
\sum_{k=0}^n (-1)^{k} \frac{\binom{n}{k}}{\binom{b+k}{k}} = \frac{b}{n+b}
\end{equation}
and combine with \eqref{main_id11}.
\end{proof}

\begin{corollary}
For $n\in\mathbb N_0$ and $b\in\mathbb C\setminus\mathbb Z^{-}$ we have
\begin{equation}\label{cor_id3}
\sum_{k=1}^n \binom{n}{k} (-1)^{k+1} \frac{k}{(k+b)^2} = \frac{H_{n+b} - H_{b}}{\binom{n+b}{n}}.
\end{equation}
In particular, 
\begin{equation}
\sum_{k=1}^n \binom{n}{k} (-1)^{k+1} \frac{k}{(k+n)^2} = \frac{H_{2n} - H_{n}}{\binom{2n}{n}}.
\end{equation}
\end{corollary}
\begin{proof}
Using the fact that the right hand-side of \eqref{cor_id2} is the binomial transform of $H_{n+b}/\binom{n+b}{n}$ we get
\begin{equation*}
\sum_{k=0}^n \binom{n}{k} (-1)^k \left ( \frac{b}{k+b} H_{b} - \frac{k}{(k+b)^2}\right ) = \frac{H_{n+b}}{\binom{n+b}{n}}. 
\end{equation*}
Combine this with 
\begin{equation}\label{binomial_frac_id}
\sum_{k=0}^n \binom{n}{k} (-1)^k \frac{1}{k+b} = \frac{1}{b \binom{n+b}{b}} 
\end{equation}
and the proof is completed.
\end{proof}

\begin{remark}
Comparing \eqref{cor_id3} with \eqref{cor_id1} we see that
\begin{equation*}
\sum_{k=0}^n \binom{n}{k} (-1)^k \frac{b}{(b+k)^2} = \frac{H_{n+b}-H_{b-1}}{\binom{n+b}{n}}
\end{equation*}
and
\begin{equation*}
\sum_{k=0}^n \binom{n}{k} (-1)^k \frac{k}{(b+k)^2} = \frac{-H_{n+b}+H_{b}}{\binom{n+b}{n}},
\end{equation*}
which upon addition again yield \eqref{binomial_frac_id}.
\end{remark}

\begin{theorem}
If $n\in\mathbb Z^+$ and $b\in\mathbb C\setminus\mathbb Z^{-}$, then
\begin{equation}\label{eq.hzbwh6o}
\sum_{k = 0}^n {( - 1)^{k - 1} \frac{{\binom{{n}}{k}}}{{\binom{{b + k}}{k}}}H_k } 
= \frac{b}{{n + b}}\left( {H_{n + b} - H_b } \right) + \frac{n}{{\left( {n + b} \right)^2 }}.
\end{equation}
\end{theorem}
\begin{proof}
Set $c=b$ in~\eqref{main_id2} and use~\eqref{eq.twi4acl}.
\end{proof}

\begin{theorem}
If $n\in\mathbb N_0$ and $b\in\mathbb C\setminus\mathbb Z^{-}$, $b\ne 0$, then
\begin{equation}\label{eq.ufus43q}
\sum_{k = 0}^n \binom{{n}}{k} \frac{(- 1)^k}{(k + b)^3} = \frac{1}{2n + 2}\binom{{n + b}}{{n + 1}}^{- 1} \left( {\left( {H_{n + b}  
- H_{b - 1} } \right)^2 - H_{b - 1}^{(2)} + H_{n + b}^{(2)} } \right).
\end{equation}
In particular,
\begin{equation}\label{Bai_id}
\sum_{k = 0}^n \binom{{n}}{k} \frac{(- 1)^k}{(k + 1)^3} = \frac{1}{2n + 2}\left( {H_{n + 1}^2 + H_{n + 1}^{(2)} } \right).
\end{equation}
\end{theorem}
\begin{proof}
Write~\eqref{cor_id1} as
\begin{equation*}
\sum_{k=0}^n (-1)^k \frac{\binom{n}{k}}{(b+k)^2} = \frac{1}{n+1} \frac{\psi(n + b + 1)-\psi(b)}{\binom{n + b}{n + 1}}
\end{equation*}
and differentiate with respect to $b$ to obtain
\begin{equation*}
\begin{split}
&\sum_{k = 0}^n {\frac{{( - 1)^k }}{{(k + b)^3 }}\binom{{n}}{k}} \\
&\qquad = \frac{1}{2n + 2} \binom{{n + b}}{{n + 1}}^{ - 1} \left( {\psi (n + b + 1) - \psi (b)} \right)^2 \\
&\qquad\quad - \frac{1}{{2n + 2}}\binom{{n + b}}{{n + 1}}^{ - 1} \left( {\psi _1 (n + b + 1) - \psi _1 (b)} \right),
\end{split}
\end{equation*}
where $\psi _1 (x)$ is the trigamma function defined by
\begin{equation*}
\psi _1 (x) = \frac{d}{{dx}}\psi (x) = \sum_{k = 0}^\infty \frac{1}{(k + x)^2}
\end{equation*}
and $H^{(2)}_r$ is the $r^{\,th}$ second order harmonic number,
\begin{equation*}
H_r^{(2)} = \sum_{j = 1}^r \frac{1}{j^2}.
\end{equation*}
The result follows upon using the fact that 
\begin{equation}\label{eq.kmcmfjf}
\psi \left( {x + 1} \right) - \psi \left( {y + 1} \right) = H_x  - H_y
\end{equation}
and
\begin{equation}\label{eq.ezvh9zc}
\psi_1 \left( {x + 1} \right) - \psi_1 \left( {y + 1} \right) = H_y^{(2)}  - H_x^{(2)},
\end{equation}
for $x,y\in\mathbb C\setminus\mathbb Z^{-}$.
\end{proof}

\begin{theorem}
If $n\in\mathbb N_0$ and $b\in\mathbb C\setminus\mathbb Z^{-}$, then
\begin{gather}
\sum_{k = 0}^n \left( { - 1} \right)^{k - 1} \binom{{n}}{k}H_{k + b}^{(2)} = \frac{{H_{n + b} - H_b }}{{n\binom{{n + b}}{n}}},\quad n\ne 0\label{eq.soe68bj},\\
\sum_{k = 1}^n \left( { - 1} \right)^{k - 1} \binom{{n}}{k}\frac{{H_{k + b} - H_b }}{{k\binom{{k + b}}{k}}} = H_{n + b}^{(2)} - H_b^{(2)}\label{eq.lht3ics},\\
\sum_{k = 1}^n \left( { - 1} \right)^{k - 1} \binom{{n}}{k}\frac{{H_{k + b} }}{{k\binom{{k + b}}{k}}} 
= \left( {H_{n + b} - H_b } \right)H_b + H_{n + b}^{(2)} - H_b^{(2)}\label{eq.qr5oi89} .
\end{gather}
\end{theorem}
\begin{proof}
Differentiate~\eqref{main_id22} with respect to $b$ to get~\eqref{eq.soe68bj}. Identity~\eqref{eq.lht3ics} is the inverse of~\eqref{eq.soe68bj}. Identity~\eqref{eq.qr5oi89} is obtained by using~\eqref{eq.auhzvb5} in~\eqref{eq.lht3ics}.
\end{proof}

\begin{theorem}
If $n\in\mathbb Z^+$ and $b\in\mathbb C\setminus\mathbb Z^{-}$, then
\begin{equation}\label{eq.iauweap}
\begin{split}
\sum_{k = 0}^n {( - 1)^{k - 1} \frac{{\binom{{n}}{k}}}{{\binom{{b + k}}{b}}}H_k H_{k + b} }  
&= \frac{1}{{n + b}}\left( {bH_b - \frac{n}{{n + b}}} \right)\left( {H_{n + b}  - H_b } \right)\\
&\qquad + \frac{b}{{n + b}}\left( {H_{n + b}^{(2)}  - H_b^{(2)} } \right) + \frac{n}{{\left( {n + b} \right)^2 }}\left( {\frac{2}{{n + b}} + H_b } \right).
\end{split}
\end{equation}
\end{theorem}
\begin{proof}
Write~\eqref{eq.hzbwh6o} as
\begin{equation*}
\sum_{k = 0}^n {( - 1)^{k - 1} \frac{{\binom{{n}}{k}}}{{\binom{{b + k}}{k}}}H_k } 
= \frac{b}{{n + b}}\left( {\psi(n + b + 1) - \psi(b + 1) } \right) + \frac{n}{{\left( {n + b} \right)^2 }}
\end{equation*}
and differentiate with respect to $b$, using~\eqref{eq.u72pdmn}. Use~\eqref{eq.hzbwh6o} again in simplifying the left hand side of the resulting expression.
\end{proof}

In particular, for all positive integers $n$, we have
\begin{equation}\label{id_Hk_squared}
\sum_{k = 0}^n (- 1)^k \binom{n}{k} H_k^2 = \frac{{H_n }}{n} - \frac{2}{n^2}
\end{equation}
and
\begin{align}
&\sum_{k = 0}^n (- 1)^{k - 1} \frac{\binom{n}{k}}{\binom{{n + k}}{k}} H_k H_{k + n} \nonumber \\
&\qquad = \frac{1}{2}\left( H_{2n}^{(2)} - H_n^{(2)} \right) + \frac{1}{2} H_{2n} \left( H_{n} - \frac{1}{2n} \right)
- \frac{1}{2} H_{n} H_{n-1} + \frac{1}{4n^2}.
\end{align}

\begin{corollary}
If $n$ is a positive integer, then
\begin{equation}\label{eq.vsafq0i}
\sum_{k=1}^n (- 1)^{k - 1} \binom{n}{k} \frac{H_k}{k} = H_n^{(2)}.
\end{equation}
\end{corollary}
\begin{proof}
Treat identity \eqref{id_Hk_squared} as the binomial transform of the sequence $H_n^2$. The inverse relation yields
\begin{equation*}
\sum_{k=1}^n (- 1)^{k} \binom{n}{k} \frac{H_k}{k} + 2 \sum_{k=1}^n \binom{n}{k} \frac{(-1)^{k-1}}{k^2} = H_n^{2}.
\end{equation*}
But (see Boyadzhiev's book \cite[p. 64]{Boyadzhiev})  
\begin{equation*}
\sum_{k=1}^n \binom{n}{k} \frac{(-1)^{k-1}}{k^2} = \frac{1}{2}\left ( H_n^2 + H_n^{(2)}\right )
\end{equation*}
and the proof is completed.
\end{proof}
\begin{remark}
Identity~\eqref{eq.vsafq0i} can also be obtained by taking the limit of~\eqref{main_id11b} as $b$ approaches zero as well as simply setting $b=0$ in~\eqref{eq.lht3ics} or~\eqref{eq.qr5oi89}.
\end{remark}

\section{Some identities involving odd harmonic numbers}

This section contains some combinatorial identities involving odd harmonic numbers $O_n$, which are defined by
\begin{equation*}
O_n = \sum_{k = 1}^n \frac{1}{2k-1},\quad O_0=0.
\end{equation*}
Obvious relations between harmonic numbers $H_n$ and odd harmonic numbers $O_n$ are given by
\begin{equation}\label{eq.h9wjxs0}
H_{2n} = \frac{1}{2} H_n + O_n \qquad \text{and} \qquad H_{2n - 1} = \frac{1}{2}H_{n - 1} + O_n.
\end{equation}
Additional relations are contained in the next lemma. 

\begin{lemma}\label{lem.czxfdu7}
If $n$ is an integer, then
\begin{gather}
H_{n - 1/2} =2O_n - 2\ln 2,\\
H_{n - 1/2} - H_{ - 1/2} = 2O_n,\label{eq.plh634k} \\
H_{n - 1/2} - H_{1/2} = 2\left( {O_n  - 1} \right),\\
H_{n + 1/2} - H_{ - 1/2} = 2O_{n + 1},\label{eq.ivi1ex5} \\
H_{n + 1/2} - H_{1/2} = 2\left( {O_{n + 1}  - 1} \right),\\
H_{n + 1/2} - H_{n - 1/2} = \frac{2}{{2n + 1}},\\
H_{n - 1/2} - H_{-3/2} = 2\left( {O_n  - 1} \right)\label{eq.pobmr6h},\\
H_{n + 1/2} - H_{-3/2} = 2\left( {O_{n + 1}  - 1} \right).
\end{gather}
\end{lemma}
\begin{proof}
Use~\eqref{eq.cx26aju} as the definition of the harmonic numbers for all complex $n$ (excluding zero and the negative integers) and use the known result for the digamma function at half-integer arguments \cite[Eq. (51)]{Srivastava}, namely,
\begin{equation*}
\psi (n + 1/2) = - \gamma - 2\ln 2 + 2\sum_{k = 1}^n \frac{1}{2k - 1}.
\end{equation*}
\end{proof}

\begin{lemma}[{Gould~\cite[Identities Z.45 and Z.51]{Gould}}]
If $r$ and $s$ are integers such that $0\le s\le r$, then
\begin{equation}\label{eq.cwrdtmu}
\binom{{r + 1/2}}{s} = \binom{{2r + 1}}{{2s}}\binom{{r}}{s}^{ - 1} 2^{ - 2s} \binom{{2s}}{s}
\end{equation}
and
\begin{equation}\label{eq.s0t6h30}
\binom{{r - 1/2}}{s} = \binom{{r}}{s}\binom{{2r - 2s}}{{r - s}}^{ - 1} 2^{ - 2s} \binom{{2r}}{r}.
\end{equation}
We also have
\begin{gather}
\binom{{r}}{{1/2}} = \dfrac{{2^{2r + 1} }}{{\pi \binom{{2r}}{r}}},\quad\text{\cite[Identity Z.48]{Gould}}\label{eq.dtwjzd8},\\
\binom{{r}}{{ - 1/2}} = \dfrac{{2^{2r + 1} }}{{\pi (2r + 1)\binom{{2r}}{r}}}\label{eq.q3uie78},
\end{gather}
and
\begin{equation}\label{eq.cm3vohl}
\binom{{r - 1/2}}{{r + 1}} =  - \frac{1}{{r + 1}}\binom{{2r}}{r}\frac{1}{{2^{2r + 1} }}.
\end{equation}
\end{lemma}
\begin{proof}
Identity~\eqref{eq.q3uie78} follows from the fact that
\begin{equation*}
\binom{{r}}{{ - 1/2}} = \frac{{r!}}{{( - 1/2)!(r + 1/2)!}},
\end{equation*}
while~\eqref{eq.cm3vohl} is a consequence of
\begin{equation}
\binom{{r - 1/2}}{{r + 1}} = - \frac{1}{2}\frac{{(r - 1/2)!}}{{(r + 1)!\sqrt \pi }}.
\end{equation}
\end{proof}

\begin{theorem}\label{main_odd1}
If $n$ is a non-negative integer, then
\begin{gather}
\sum_{k = 0}^n \binom{n}{k} \frac{(- 1)^k}{\left( {2k + 1} \right)^2} 
= \frac{2^{2n + 1}}{n+1} \frac{O_{n+1}}{\binom{2(n + 1)}{n + 1}}, \label{eq.main_odd1_1} \\
\sum_{k = 0}^n \binom{n}{k} (- 1)^k \frac{2^{2k + 1}}{k+1} \frac{O_{k+1}}{\binom{2(k + 1)}{k + 1}} = \frac{1}{(2n+1)^2},\\
\sum_{k = 0}^n \binom{{n}}{k} \frac{(- 1)^{k - 1}}{\left( {2k - 1} \right)^2} = 2^{2n} \frac{{O_n - 1}}{\binom{{2n}}{n}} \label{eq.cqklzqh},
\end{gather}
and
\begin{equation}\label{eq.fhj2wwz}
\sum_{k = 0}^n \binom{{n}}{k}(- 1)^{k - 1} 2^{2k} \frac{O_k}{\binom{{2k}}{k}} = \frac{2n}{(2n - 1)^2}.
\end{equation}
\end{theorem}
\begin{proof}
The first identity is obtained by setting $b=1/2$ in~\eqref{cor_id1} and using~\eqref{eq.ivi1ex5} and~\eqref{eq.cwrdtmu}; 
while the second is its inverse transform. Identity~\eqref{eq.cqklzqh} is obtained by setting $b=-1/2$ in~\eqref{cor_id1} 
and using~\eqref{eq.pobmr6h} and~\eqref{eq.cm3vohl}. To prove identity~\eqref{eq.fhj2wwz} use the inverse binomial transform 
of \eqref{eq.cqklzqh} to get
\begin{equation*}
\sum_{k = 0}^n \binom{{n}}{k}(- 1)^{k - 1} 2^{2k} \frac{O_k}{\binom{{2k}}{k}} 
- \sum_{k = 0}^n \binom{{n}}{k}(- 1)^{k - 1} \frac{2^{2k}}{\binom{{2k}}{k}} = \frac{1}{(2n-1)^2}.
\end{equation*}
But
\begin{equation}\label{eq.xyz1}
\sum_{k = 0}^n \binom{{n}}{k}(- 1)^{k - 1} \frac{2^{2k}}{\binom{{2k}}{k}} = \frac{1}{2n-1}
\end{equation}
and the proof of ~\eqref{eq.fhj2wwz} is completed.
\end{proof}

\begin{theorem}\label{main_odd2}
If $n$ is a non-negative integer, then
\begin{gather}
\sum_{k = 0}^n {( - 1)^{k + 1} \binom{{n}}{k}\frac{{2k - 1}}{{2^{2(k - 1)} }}\binom{{2(k - 1)}}{{k - 1}}O_k }  = \frac{1}{{2^{2(n - 1)} }}\binom{{2(n - 1)}}{{n - 1}}\left( {\frac{{2n}}{{2n - 1}} - O_n } \right)\label{eq.nh1vubi},\\
\sum_{k = 0}^n {( - 1)^{k + 1} \binom{{n}}{k}\frac{{2k + 1}}{{2^{2k} }}\binom{{2k}}{k}O_{k + 1} }  = \frac{1}{{2n - 1}}\frac{{\binom{{2n}}{n}}}{{2^{2n} }}\left( {\frac{{4n - 1}}{{2n - 1}} - O_n } \right)\label{eq.nzxpc5c},\\
\sum_{k = 0}^n {( - 1)^{k + 1} \binom{{n}}{k}2^{ - 2k - 2} \binom{{2(k + 1)}}{{k + 1}}(O_{k + 1} - 1) }  
= \frac{2^{ - 2n - 2}}{{2n + 1}} \binom{{2(n + 1)}}{{n + 1}}\left( {O_{n + 1} -\frac{1}{{2n + 1}}} \right)\label{eq.pj6av0a}, \\
\sum_{k = 0}^n {( - 1)^{k + 1} \binom{{n}}{k}2^{ - 2k} \binom{{2k}}{k}(O_k - 1) }  
= 2^{ - 2n} \binom{{2n}}{n}\left( {\frac{{2n}}{{2n + 1}} + O_{n + 1}} \right)\label{eq.zxdl17v}.
\end{gather}
\end{theorem}
\begin{proof}
These results follow from \eqref{main_id2}. To obtain~\eqref{eq.nh1vubi}, set $b=-1$, $c=-1/2$ in~\eqref{main_id2} and use~\eqref{eq.plh634k} and~\eqref{eq.q3uie78}. Identity~\eqref{eq.nzxpc5c} comes from setting $b=0$, $c=-1/2$ in~\eqref{main_id2} and using~\eqref{eq.ivi1ex5},~\eqref{eq.dtwjzd8} and~\eqref{eq.q3uie78}.
\end{proof}

\begin{theorem}\label{main_odd3}
If $n$ is a positive integer, then
\begin{gather}
\sum_{k = 1}^n \binom{{n}}{k} ( - 1)^{k + 1} O_k = \frac{{2^{2n - 1} }}{{n \binom{{2n}}{n}}} \label{eq.jm6rck7},\\
\sum_{k = 1}^n \binom{{n}}{k} ( - 1)^{k + 1} \frac{2^{2k-1}}{k \binom{{2k}}{k}} = O_n \label{eq.u1t9s6r},\\
\sum_{k = 0}^n \binom{{n}}{k} ( - 1)^{k + 1} O_{k + 1}  = \frac{{2^{2n - 1} }}{{n(2n + 1)\binom{{2n}}{n}}} \label{eq.fqnumdk},\\
\sum_{k = 1}^n \binom{{n}}{k} ( - 1)^{k + 1} \frac{{2^{2k - 1} }}{k(2k + 1) \binom{{2k}}{k}} = O_{n + 1} - 1  \label{eq.b357dyz}.
\end{gather}
\end{theorem}
\begin{proof}
Write~\eqref{main_id22} as
\begin{equation*}
\sum_{k = 0}^n {( - 1)^{k + 1} \binom{{n}}{k}\left( {H_{k + b} - H_c } \right)} = \frac{1}{{n\binom{{n + b}}{n}}}.
\end{equation*}
Evaluate at $(b,c)=(-1/2,-1/2)$ and at $(b,c)=(1/2,-1/2)$ to obtain~\eqref{eq.jm6rck7} and~\eqref{eq.fqnumdk}. Identity~\eqref{eq.u1t9s6r} is the binomial transform of~\eqref{eq.jm6rck7} while~\eqref{eq.b357dyz} is the transform of~\eqref{eq.fqnumdk}.
\end{proof}

\begin{theorem}\label{main_odd4}
If $n$ is a non-negative integer, then
\begin{gather}
\sum_{k = 1}^n \binom{{n}}{k}\frac{{(- 1)^{k + 1} k}}{(2k + 1)^2} = \frac{{2^{2n - 1} (O_{n + 1} - 1)}}{(2n + 1)\binom{{2n}}{n}}, \\
\sum_{k = 1}^n \binom{{n}}{k}\frac{{(- 1)^{k + 1} k}}{(2k - 1)^2} = \frac{{2^{2n - 1} O_n}}{\binom{{2n}}{n}}, \label{eq.main_odd4_2} \\
\sum_{k = 1}^n \binom{{n}}{k} ( - 1)^{k - 1} \frac{{2^{2k - 1} (O_{k + 1} - 1)}}{(2k + 1)\binom{{2k}}{k}} = \frac{n}{(2n + 1)^2}, \\
\sum_{k = 1}^n \binom{{n}}{k} ( - 1)^{k + 1}\frac{2^{2k - 1} O_k }{\binom{{2k}}{k}} = \frac{n}{(2n - 1)^2} \label{eq.hc2fi1o}.
\end{gather}
\end{theorem}
\begin{proof}
The first two results follow from \eqref{cor_id3}. The other two identities are the inverse binomial transforms of the former.
\end{proof}

\begin{remark}
Identity \eqref{eq.hc2fi1o} is a rediscovery of identity \eqref{eq.fhj2wwz}. Also, combining the proof of \eqref{eq.fhj2wwz}
with \eqref{eq.hc2fi1o} yields \eqref{eq.xyz1}.
\end{remark}

\begin{theorem}\label{main_odd5}
If $n$ is a non-negative integer, then
\begin{gather}
\sum_{k = 0}^n \binom{{n}}{k} (- 1)^{k - 1} \frac{2^{2k} H_k}{(2k+1) \binom{{2k}}{k}} 
= \frac{2O_{n + 1} }{2n + 1} - \frac{2}{(2n + 1)^2}\label{eq.zmrbc4e}, \\
\sum_{k = 0}^n \binom{{n}}{k} (- 1)^k \frac{2^{2k} H_k}{\binom{{2k}}{k}} = \frac{2O_n}{{2n - 1}} - \frac{4n}{(2n - 1)^2} \label{eq.mtof7yp}, \\
\sum_{k = 0}^n \binom{{n}}{k} (- 1)^k \frac{O_{k+1}}{2k+1} 
= \frac{2^{2n+1}}{n+1} \frac{O_{n+1}}{\binom{2n+2}{n+1}} - \frac{2^{2n-1}}{2n+1} \frac{H_n}{\binom{2n}{n}} \label{eq.xxyyzz1}, \\
\sum_{k = 0}^n \binom{{n}}{k} (- 1)^k \frac{O_{k}}{2k-1} 
= \frac{2^{2n-1}}{\binom{2n}{n}} \left ( H_n - 2O_{n} \right ), \label{eq.xxyyzz2}\\
\sum_{k = 1}^n {\left( { - 1} \right)^{k - 1} \binom{{n}}{k}\frac{{2^{2k} }}{{\binom{{2k}}{k}}}H_{2k} }  = \frac{{4n}}{{\left( {2n - 1} \right)^2 }} - \frac{{O_n }}{{2n - 1}}\label{eq.yq8ibvs}.
\end{gather}
\end{theorem}
\begin{proof}
The first two results follow from \eqref{eq.hzbwh6o}. Identity \eqref{eq.xxyyzz1} is a consequence of the inverse binomial relation 
of \eqref{eq.zmrbc4e} in conjunction with identity \eqref{eq.main_odd1_1}. Identity \eqref{eq.xxyyzz2} is a consequence of the inverse binomial relation of \eqref{eq.mtof7yp} in conjunction with identity \eqref{eq.main_odd4_2}. Identity~\eqref{eq.yq8ibvs} is a consequence of~\eqref{eq.hc2fi1o} and~\eqref{eq.mtof7yp} on account of the first identity in~\eqref{eq.h9wjxs0}.
\end{proof}

\begin{theorem}
If $n$ is a non-negative integer, then
\begin{equation}
\sum_{k = 0}^n {\binom{{n}}{k}\frac{{\left( { - 1} \right)^{k - 1} }}{{\left( {2k - 1} \right)^3 }}}  = \dfrac{{2^{2n - 1} }}{{\binom{{2n}}{n}}}\left( {\left( {O_n  - 1} \right)^2  + O_n^{(2)}  + 1} \right).
\end{equation}
\end{theorem}
\begin{proof}
Set $b=-1/2$ in~\eqref{eq.ufus43q} and use Lemma~\ref{lem.czxfdu7} and identity~\ref{eq.s0t6h30}. Use also 
\begin{equation}\label{eq.cqhkqwq}
H_{n - 1/2}^{(2)}  =  - 2\zeta (2) + 4O_n^{(2)}
\end{equation}
which follows from the known value $H_{-1/2}^{(2)}=-2\zeta(2)$ since
\begin{equation*}
H_{n - 1/2}^{(2)}  - H_{ - 1/2}^{(2)}  = 4O_n^{(2)},
\end{equation*}
where $O_n^{(2)}$ is the $n^{th}$ second order odd harmonic number defined by
\begin{equation*}
O_n^{(2)}  = \sum_{k = 1}^n {\frac{1}{{\left( {2k - 1} \right)^2 }}}.
\end{equation*}
\end{proof}
\begin{theorem}
If $n$ is a positive integer, then
\begin{gather}
\sum_{k = 1}^n {\left( { - 1} \right)^{k - 1} \binom{{n}}{k}O_k^{(2)} }  = \frac{{2^{2n - 1} }}{{\binom{{2n}}{n}}}\frac{{O_n }}{n},\label{eq.awov07j}\\
\sum_{k = 1}^n {\left( { - 1} \right)^{k - 1} \binom{{n}}{k}\frac{{2^{2k - 1} }}{{\binom{{2k}}{k}}}\frac{{O_k }}{k}}  = O_n^{(2)}\label{eq.y491lmb} .
\end{gather}
\end{theorem}
\begin{proof}
Set $b=-1/2$ in~\eqref{eq.soe68bj} and use Lemma~\ref{lem.czxfdu7} and also~\eqref{eq.s0t6h30} and~\eqref{eq.cqhkqwq} to obtain~\eqref{eq.awov07j}. Identity~\eqref{eq.y491lmb} is the inverse of~\eqref{eq.awov07j}.
\end{proof}
\begin{theorem}
If $n$ is a non-negative integer, then
\begin{gather}
\sum_{k = 0}^n {\left( { - 1} \right)^{k - 1} \binom{{n}}{k}\frac{{2^{2k} }}{{\binom{{2k}}{k}}}H_k O_k }  = \frac{{8n}}{{\left( {2n - 1} \right)^3 }} - \frac{{4n\,O_n }}{{\left( {2n - 1} \right)^2 }} - \frac{{2O_n^{(2)} }}{{\left( {2n - 1} \right)}},\\
\sum_{k = 0}^n {\left( { - 1} \right)^k \binom{{n}}{k}\frac{{2^{2k} }}{{\binom{{2k}}{k}}}H_k }  = \frac{{2O_n }}{{\left( {2n - 1} \right)}} - \frac{{4n}}{{\left( {2n - 1} \right)^2 }}\label{eq.wsi7oju}.
\end{gather}
\end{theorem}
\begin{proof}
Set $b=-1/2$ in~\eqref{eq.iauweap}. Use Lemma~\ref{lem.czxfdu7},~\eqref{eq.s0t6h30} and~\eqref{eq.cqhkqwq}. Equate rational coefficients from both sides. Identity~\eqref{eq.wsi7oju} is a rediscovery of~\eqref{eq.mtof7yp}.
\end{proof}

\section{Final comments}

In these notes, we have demonstrated how a combinatorial identity attributed to Frisch can be used to prove a series of 
harmonic number and odd harmonic number identities. Some of the results that we derived are known and here we provide 
final comments and give further references. \\
We begin by pointing out that identity \eqref{id_Hk_squared} can also be found in a recent article by Batir \cite[Identity 18]{Batir}.
Next, identity \eqref{binomial_frac_id} can be restated as
\begin{equation*}
\sum_{k=0}^n \binom{n}{k} (-1)^k \frac{b}{k+b} = \prod_{k=1}^n \frac{k}{b + k}. 
\end{equation*}
This identity is known and two probabilistic proofs were given recently by Peterson \cite{Peterson} and Nakata \cite{Nakata}.
Peterson \cite{Peterson} writes \eqref{cor_id1} in the form
\begin{equation*}
\sum_{k=0}^n \binom{n}{k} (-1)^k \left (\frac{b}{b+k}\right )^2 = \prod_{k=1}^n \frac{k}{b + k} \left (1 + \sum_{k=1}^n \frac{b}{b+k}\right ),
\end{equation*}
and also states an expression for the generalization involving an additional parameter $m$
\begin{equation}\label{Peterson}
\sum_{k=0}^n \binom{n}{k} (-1)^k \left (\frac{b}{b+k}\right )^m, \qquad m\geq 1.
\end{equation}
In 2019, Bai and Luo \cite{Bai} derived a new expression for Peterson's identity \eqref{Peterson} 
using a partial fraction decomposition and involving generalized harmonic numbers. As applications of their main result, they
stated some harmonic number identities. One of their special cases is our identity \eqref{Bai_id}. \\

It is remarkable that a simple expression for \eqref{Peterson} is also hidden in Frisch's identity \eqref{Id_Frisch}.
Indeed, setting $c=1$ in \eqref{Id_Frisch} yields \eqref{binomial_frac_id}, i.e.,
\begin{equation*}
\sum_{k=0}^n \binom{n}{k} \frac{(-1)^k}{b+k} = \frac{1}{b \binom{n+b}{b}}. 
\end{equation*}
Differentiating both sides $m$ times w.r.t. $b$ gives
\begin{equation}
\sum_{k=0}^n \binom{n}{k} \frac{(-1)^k}{(b+k)^m} = \frac{(-1)^{m-1}}{(m-1)!}\frac{d^{m-1}}{db^{m-1}}\frac{1}{b \binom{n+b}{b}}, 
\end{equation}
which shows such a simple expression.

\end{document}